\documentclass[12pt]{amsart}

\usepackage{amsmath,amssymb}

\newtheorem{thm}{Theorem}
\newtheorem{lem}[thm]{Lemma}
\newtheorem{prop}[thm]{Proposition}
\newtheorem{cor}[thm]{Corollary}

\newcommand{\lcm}{{\operatorname{lcm}}}
\newcommand{\rad}{{\operatorname{rad}}}
\newcommand{\bj}{{\bf j}}
\newcommand{\bJ}{{\bf J}}

\newcommand{\Z}{{\mathbb Z}} 
\newcommand{\Q}{{\mathbb Q}} 

\newcommand{\ord}[2]{{\ell}_{#1}({#2})}

\title[On a problem of Arnold]{On a problem of Arnold: 
  the average   multiplicative order of a given integer}

\author{P\"ar Kurlberg}
\address{Department of Mathematics, Royal Institute of Technology,
 SE-100 44 Stockholm, Sweden}
 \email{kurlberg@math.kth.se}

\author{Carl Pomerance}
\address{Mathematics Department\\
Dartmouth College\\
Hanover, NH 03755-3551\\
U.S.A.}
\email{carl.pomerance@dartmouth.edu}

\thanks{P.K. was partially supported by grants from the G\"oran
Gustafsson Foundation, the Knut and Alice Wallenberg foundation, the
Royal Swedish Academy of Sciences, and the
Swedish Research Council.  C.P. was supported by NSF grant numbers DMS-0703850,
DMS-1001180. }

\begin{document}
\date{August 25, 2010}

\begin{abstract}
  For $g,n$ coprime integers, let $\ell_{g}(n)$ denote the
  multiplicative order of $g$ modulo $n$.  Motivated by a conjecture
  of Arnold, we study the average of $\ell_{g}(n)$ as $n \leq x$
  ranges over integers coprime to $g$, and $x$ tending to infinity.
  Assuming the Generalized Riemann Hypothesis, we show that this
  average is essentially as large as the average of the Carmichael
  lambda function.
We  also determine the asymptotics of the average of $\ell_{g}(p)$ as
  $p \leq x$ ranges over primes.
\end{abstract}

\maketitle

\section{Introduction}

Given coprime integers $g,n$ with $n>0$ and $|g|>1$, 
let $\ord{g}{n}$ denote the multiplicative
order of $g$ modulo $n$, i.e., the smallest integer $k \geq 1$ such
that $g^{k} \equiv 1 \mod n$.  For $x \geq 1$ an integer let
$$
T_g(x) 
:=
\frac{1}{x}
\sum_{
\substack{
n \leq x\\
(n,g)=1    
}}
\ord{g}{n},
$$
essentially the average multiplicative order of $g$.
In \cite{arnold05-number-theoretical}, Arnold conjectured that if
$|g|>1$, then
$$
T_g(x) \sim c(g) \frac{x}{\log x},
$$
as $x \to \infty$, for some constant $c(g)>0$.
However, in \cite{shparlinski07-dynamical} Shparlinski showed that
if the Generalized Riemann Hypothesis\footnote{What is
  needed is that the Riemann hypothesis holds for Dedekind zeta functions
  $\zeta_{K_n}(s)$ for all $n>1$, where $K_{n}$ is the Kummer
  extension $\Q({\rm e}^{2\pi i/n}, g^{1/n})$.} (GRH) is true, then 
$$
T_g(x)
\gg
\frac{x}{\log x}
\exp \left(
C(g) ( \log \log \log x)^{3/2}
\right),
$$
where $C(g)>0$.
He also suggested that it should be possible to obtain, again assuming
GRH, a lower bound of the form 
$$
T_g(x)
\ge
\frac{x}{\log x} \exp \left(  ( \log \log \log x)^{2+o(1)} \right),
$$
as $x\to\infty$.

Let
\begin{equation}
\label{eq:B}
B={\rm e}^{-\gamma}\prod_p\left(1-\frac1{(p-1)^2(p+1)}\right)
=0.3453720641\dots,
\end{equation}
the product being over primes, and where $\gamma$ is the Euler--Mascheroni
constant.
The principal aim of this paper is to prove the following result.
\begin{thm}
\label{thm:main}
Assuming the GRH, 
$$
T_g(x)
=
\frac{x}{\log x}
\exp \left( 
\frac{B \log \log x}{\log \log \log x}(1+o(1))
\right)
$$
as $x\to\infty$, uniformly in $g$ with $1<|g|\le\log x$.
The upper bound implicit in this result holds unconditionally.
\end{thm}

Let $\lambda(n)$ denote the exponent of the group $(\Z/n\Z)^\times$.
Commonly known as Carmichael's function, we have $\ord{g}{n}\le\lambda(n)$
when $(g,n)=1$, so we immediately obtain that
$$
T_g(x) \leq \frac1x \sum_{n\le x} \lambda(n),
$$
and it is via this inequality that we are able to unconditionally
establish the upper bound implicit in Theorem~\ref{thm:main}.
Indeed, in \cite{carmichael-lambda}, Erd\H os, Pomerance, and Schmutz 
determined the average order  of $\lambda(n)$ showing that, 
as $x\to\infty$,  
\begin{equation}
\label{eq:lambda}
\frac1x \sum_{n\le x} \lambda(n)
=
\frac{x}{\log x}
\exp \left( 
\frac{B \log \log x}{\log \log \log x}(1+o(1))
\right).
\end{equation}

Theorem~\ref{thm:main} thus shows under assumption of the GRH that
 the mean values of $\lambda(n)$ and $\ord{g}{n}$ are of
a similar order of magnitude.  We know, on assuming the GRH, that
$\lambda(n)/\ord{g}{n}$
is very small for almost all $n$
(e.g., see
\cite{artincat,li-pomerance-artin-for-composite}; in the latter paper Li and
Pomerance in fact showed that $\lambda(n)/\ord{g}{n} \leq (\log
n)^{o(\log \log \log n)}$ as $n\to\infty$ on a set of asymptotic density~1), so 
perhaps Theorem~\ref{thm:main} is not very surprising. 
{\em However}, in \cite{carmichael-lambda} it was also shown that the
normal order of $\lambda(n)$ is quite a bit smaller than the average
order: there exists a subset $S$ of the positive integers, of
asymptotic density~$1$, such that for $n \in S$ and $n\to\infty$,
$$
\lambda(n) = 
\frac{n}{(\log n)^{\log \log \log n+A +(\log \log \log n)^{-1+o(1)}}},
$$
where $A>0$ is an explicit constant.
Thus the main contribution to the average of $\lambda(n)$ comes from a
{\em density-zero subset} of the integers, and to obtain our result on
the average multiplicative order, we must show that $\ord{g}{n}$ is
large for many $n$ for which $\lambda(n)$ is large.  

We remark that if one averages over $g$ as well, then a result like
our Theorem~\ref{thm:main} holds unconditionally.  In particular, it
follows from Luca and
Shparlinski~\cite[Theorem~6]{luca-shparlinski03-average-mult-order} that
$$
\frac1{x^2}\sum_{n\le x}\sum_{\substack{1<g<n\\ (g,n)=1}}\ord{g}{n}
=\frac{x}{\log x}\exp\left(\frac{B\log\log x}{\log\log\log x}(1+o(1))\right)
$$
as $x\to\infty$.

We also note that our methods give that Theorem~\ref{thm:main} still
holds for $g=a/b$ a rational number, with uniform error for $|a|,|b|
\leq \log x$, and $n$ ranging over integers coprime to $ab$.

\subsection{Averaging over prime moduli}
\label{subsecintro}

We shall always have the letters $p,q$ denoting prime numbers.
Given a rational number $g \neq 0,\pm 1$ and a prime $p$ not dividing
the numerator or denominator of $g$, let $\ord{g}{p}$ denote
the multiplicative order of $g$ modulo $p$.  For simplicity, when $p$ does divide
the numerator or denominator of $g$, we let $\ord{g}{p}=1$.
Further, given $k \in \Z^+$, let
$$
D_{g}(k) := [\Q(g^{1/k},{\rm e}^{2\pi i /k}):\Q]
$$
denote the degree of the Kummer extension obtained by taking the
splitting field of $X^k-g$.  
Let $\rad(k)$ denote the largest squarefree
divisor of $k$ and let $\omega(k)$ be the number of primes dividing
$\rad(k)$.

\begin{thm}
\label{thm:prime-main}
Given $g \in \Q$, $g \neq 0, \pm 1$, define
$$
c_{g} := 
\sum_{k=1}^{\infty}
\frac{\phi(k) \rad(k) (-1)^{\omega(k)}}{k^{2} D_{g}(k)}.
$$
The series for $c_g$ converges absolutely, and,
assuming the GRH, 
$$
\frac{1}{\pi(x)}
\sum_{p \leq x} \ord{g}{p} =
\frac12c_{g} \cdot x + 
O\left(\frac{x}{(\log x)^{1/2-1/\log\log\log x}}\right).
$$
Further, with $g = a/b$ where $a,b \in \Z$, the error estimate
holds uniformly for $|a|,|b| \leq x$.
\end{thm}
\noindent
This result might be compared with Pappalardi~\cite{Pa}.

Though perhaps not obvious from the definition, $c_{g}>0$ for all $g
\neq 0,\pm 1$.  In order to determine $c_{g}$, define
$$
c :=\prod_p\left(1-\frac{p}{p^3-1}\right)=0.5759599689\dots,
$$
the product being over primes; $c_{g}$ turns out to be a
positive {\em rational} multiple of $c$.  Theorem~\ref{thm:prime-main}
should be contrasted with the unconditional result of 
Luca~\cite{luca05-mean-multiplicative-orders} that
$$
\frac1{\pi(x)}\sum_{p\le x}\frac1{(p-1)^2}\sum_{g=1}^{p-1}\ord{g}{p}
= c+O(1/(\log x)^\kappa)
$$
for any fixed $\kappa>0$.
By partial summation one can then obtain
$$
\frac1{\pi(x)}\sum_{p\le x}\frac1{p-1}\sum_{g=1}^{p-1}\ord{g}{p}
\sim\frac12c\cdot x\hbox{ as }x\to\infty,
$$
a result that is more comparable to Theorem~\ref{thm:prime-main}.

To sum the series that defines $c_{g}$ we will need some further notation.
Write $g = \pm g_{0}^{h}$ where $h$ is a positive integer and
$g_{0}>0$ is not an exact power of a
rational number, and write $g_{0} = g_{1} g_{2}^{2}$ where $g_{1}$ is
a squarefree integer and $g_{2}$ is a rational.  Define 
$\Delta(g)=g_{1}$ if $g_{1} \equiv 1 \mod 4$, and 
$\Delta(g)= 4 g_{1}$ if $g_{1} \equiv 2$ or $3 \mod 4$.  
Let $e=v_2(h)$ (that is, $2^e\| h$).
For $g>0$, define $n = \lcm[2^{e+1},\Delta(g)]$.
For $g<0$, define $n = 2 g_{1}$ if $e=0$ and $g_{1} \equiv 3 \mod 4$,
or $e=1$ and $g_{1} \equiv 2 \mod 4$; let $n = \lcm[2^{e+2},\Delta(g)]$
otherwise.  

Consider the multiplicative function
$f(k) = (-1)^{\omega(k)} \rad(k)(h,k)/k^{3}$.  
We note that for $p$ prime and $j \geq 1$, 
$$
f(p^{j}) = -p^{1-3j+\min(j,v_p(h)))}.
$$
Given an integer $t\geq 1$, define $F(p,t)$ and $F(p)$ by
$$
F(p,t) := \sum_{j=0}^{t-1} f(p^{j}), \quad
F(p) := \sum_{j=0}^\infty f(p^{j})
$$
In particular, we note that if  $p \nmid h$, then
\begin{equation}
  \label{eq:Fp-if-no-h}
F(p)  =
1 - \sum_{j=1}^{\infty} p^{1-3j}  =   1- \frac{p}{p^{3}-1} .
\end{equation}

\begin{prop} 
\label{prop:finding-cg}
With notation as above, if $g<0$ and $e>0$, we have
$$
c_{g} = 
c \cdot
\prod_{p |h} \frac{F(p)}{1-\frac{p}{p^{3}-1}}
\cdot
\left( 1 -
\frac{F(2,e+1)-1}{2F(2)}
+
\prod_{p |n}  \left(1 - \frac{F(p,v_p(n))}{F(p)} \right)
\right),
$$
otherwise
$$
c_{g} = 
c \cdot
\prod_{p |h} \frac{F(p)}{1-\frac{p}{p^{3}-1}}
\cdot
\left( 1 + 
\prod_{p |n}  \left(1 - \frac{F(p,v_p(n))}{F(p)} \right)
\right).
$$
\end{prop}
For example, if $g=2$, then $h=1$, $e=0$, and $n = 8$.  Thus
$$
c_{2} = c \cdot \left(1 + 1 - \frac{F(2,3)}{F(2)} \right)
= c \cdot
\left( 2 - \frac{1-2/(2^{1})^{3}-2/(2^2)^{3}}{1-2/(8-1)} \right) 
= c \cdot \frac{159}{160}.
$$

\section{Some preliminary results}


For an integer $m\ge2$, we let $P(m)$ denote the largest prime
dividing $m$, and we let $P(1)=1$.

Given a rational number $g\ne0,\pm1$, we recall the notation $h,e,n$
described in Section~\ref{subsecintro}, and for a positive integer $k$,
we recall that $D_g(k)$
is the degree of the splitting field of $X^k-g$ over $\Q$.
We record a result of Wagstaff on $D_g(k)$,
see~\cite{wag}, Proposition~4.1 and the second
paragraph in the proof of Theorem~2.2.
\begin{prop}
\label{prop:Wagstaff}
With notation as above,
\begin{equation}
  \label{eq:kummer-degree}
D_{g}(k) =  \frac{\phi(k) \cdot k }{(k,h) \cdot \epsilon_g(k)}   
\end{equation}
where $\phi$ is Euler's function and $\epsilon_g(k)$ is defined as follows:
If $g>0$, then
$$
\epsilon_g(k) 
:=
\begin{cases}
2 & \text{if $n | k$},\\
1 & \text{if $n \nmid k$}.
\end{cases}
$$
If $g<0$, then
$$
\epsilon_g(k) 
:=
\begin{cases}
2 & \text{if $n | k$},\\
1/2 & \text{if $2 | k$ and $2^{e+1} \nmid k$},\\
1 & \text{otherwise}.
\end{cases}
$$
\end{prop}

We also record a GRH-conditional version of the Chebotarev density
theorem for Kummerian fields over $\Q$, see Hooley~\cite[Sec.~5]{hooley-artin}
and Lagarias and Odlyzko~\cite[Theorem~1]{LG}.  Let $i_g(p)=(p-1)/\ord{g}{p}$,
the index of $\langle g\rangle$ in $(\Z/p\Z)^*$ when $g\in(\Z/p\Z)^*$.
\begin{thm}
\label{thm:pnt}
Assume the GRH.
Suppose $g=a/b\ne0,\pm1$ where $a,b$ are integers of absolute value at most $x$.
For each integer $k\le x$, we have that the number of primes $p\le x$
for which $k\mid i_g(p)$ is 
$$
\frac1{D_g(k)}\pi(x)+O(x^{1/2}\log x).
$$
\end{thm}
\noindent
Note that $k\mid i_g(p)$ if and only if $x^k-g$ splits completely modulo~$p$. 

We will need the following {\em uniform} version of
\cite[Theorem~23]{power-generator}.
\begin{thm}
\label{thm:uniform-order-on-grh}
If the GRH is true, then for $x,L$ with $1\le L\le\log x$
and $g=a/b\ne0,\pm1$ where $a,b$ are integers with $|a|,|b|\le x$,
we have 
\[
\left|\left\{ p \leq x : \ord{g}{p} \leq \frac{p-1}{L}\right \}\right|
~\ll~  \displaystyle{\frac{\pi(x)}{L}\cdot\frac{h \tau(h)}{\phi(h)}}
+ 
\frac{x \log \log x }{\log^2 x},
\]
where $\tau(h)$ is the number of divisors of $h$.
\end{thm}
\begin{proof}
  Since the proof is rather similar to the proof of the main theorem
  in \cite{hooley-artin}, 
  \cite[Theorem~2]{artincat}, and \cite[Theorem~23]{power-generator},
we only give a brief outline.
We see that $\ord{g}{p} \leq (p-1)/L$ implies that 
$i_g(p)\ge L$.  Further, in the case that $p\mid ab$, where
we are defining $\ord{g}{p}=1$ and hence $i_g(p)=p-1$, the number
of primes $p$ is $O(\log x)$.  So we assume that
$p\nmid ab$.

{\em First step:}  Consider primes $p\le x$ such that 
$i_g(p) > x^{1/2} \log^{2} x$.  
Such a prime $p$ divides $a^k-b^k$ for some positive integer
$k<x^{1/2}/\log^2x$.  Since $\omega(|a^k-b^k|)\ll k\log x$,
it follows that the number of primes $p$ in this case is
$O((x^{1/2}/\log^2x)^2\log x)=O(x/\log^3x)$.

{\em Second step:} Consider primes $p$ such that $q\mid i_g(p)$ 
for some prime $q
$ in the interval $I:= [ \frac{x^{1/2}}{\log^2 x}, x^{1/2} \log^{2} x ]$. We
may bound this by considering primes 
$p \leq x$ such that $p \equiv  1~({\rm mod}~ q)$ for
some prime $q \in I$.
The Brun--Titchmarsh inequality then gives
that the number of  such primes $p$ is at most
a constant times
\begin{equation*}
\sum_{ q \in I  }
\frac{x}{\phi(q) \log(x/q)}
~\ll
~\frac{x}{\log x}
\sum_{ q \in I  }
\frac{1}{q}
~\ll
~\frac{x \log\log x}{\log^2 x}.
\end{equation*}

{\em Third step:} Now consider primes $p$ such that $q\mid i_g(p)$ 
for some prime
$q$ in the interval $[ L, \frac{x^{1/2}}{\log^2 x})$.
In this range we use Proposition~\ref{prop:Wagstaff} and
Theorem~\ref{thm:pnt} to get on the GRH that
$$
|\{ p \leq x : q \mid i_g(p)  \}|
~\ll 
~\frac{\pi(x) (q,h)}{q \phi(q)} +  x^{1/2} \log x .
$$
Summing over primes $q$, we find that the number of such $p$  is 
bounded by a constant times
\begin{equation*}
\sum_{ q \in [ L, \frac{x^{1/2}}{\log^2 x})  }
\left( 
  \frac{\pi(x)(q,h)}{q^{2}} +  x^{1/2} \log x
\right)
~\ll
~\frac{\pi(x) \omega(h)}{L}
+
\frac{x }{\log^2 x}.
\end{equation*}

{\em Fourth step:} For the remaining primes $p$, any prime divisor
$q\mid i_g(p)$ is smaller than $L$. Hence $i_g(p)$ must be divisible by
some integer $d$ in the interval $[L,L^2]$.  
By Proposition~\ref{prop:Wagstaff} and 
Theorem~\ref{thm:pnt}, assuming the GRH, we have
\begin{equation}
\label{chebotarev}
|\{ p \leq x : d \mid i_g(p)  \}|
~\le 
~2\frac{\pi(x)(d,h)}{d \phi(d)} + O( x^{1/2} \log x ).
\end{equation}
Hence the total number of such $p$  is bounded by 
\begin{equation*}
\sum_{ d \in [L,L^2]  }
\left( 
 2 \frac{\pi(x)(d,h)}{d \phi(d)} + O( x^{1/2} \log x )
\right)
~\ll
~\frac{\pi(x)}{L}\frac{h\tau(h)}{\phi(h)},
\end{equation*}
where the last estimate follows from
\begin{multline}
\label{eq:sum-estimate}
\sum_{ d \in [L,L^2]  }
\frac{(d,h)}{d \phi(d)}
\leq
\sum_{m|h} 
\sum_{\substack{ d \in [L,L^2]\\ m|d}  }
\frac{m}{d \phi(d)}
\le
\sum_{m|h}\sum_{k\ge L/m}\frac{1}{\phi(m)k\phi(k)}\\
\ll
\sum_{m|h}\frac{m}{L\phi(m)}
=\frac{h}{L\phi(h)}\sum_{m|h}\frac{m}{\phi(m)}\cdot\frac{\phi(h)}{h}
\le\frac{h\tau(h)}{L\phi(h)}.
\end{multline}
Here we used the
bound $\sum_{k \geq T} \frac{1}{k\phi(k)} \ll 1/T$ for $T>0$, 
which follows by an elementary argument 
from the bound $\sum_{k\ge T}\frac{1}{k^2}\ll 1/T$
and the identity $k/\phi(k)=\sum_{j|k}\frac{\mu^2(j)}{\phi(j)}$.
\end{proof}

\begin{cor}
\label{cor:recip}
Assume the GRH is true.  Let $m\ge2$ be an integer and $x\ge3$
a real number.  Let $y=\log\log x$ and assume that $m\le\log y/\log\log y$.  
Let $g=a/b\ne0,\pm1$ where
$a,b$ are integers with $|a|,|b|\le\exp((\log x)^{3/m})$, 
and let $h$ be as above.
Then uniformly,
$$
\sum_{\substack{p\le x\\ P(i_g(p))>m}}
\frac1p\ll{y}\left(\frac1{m}+\sum_{q\mid h,~q>m}\frac1q\right).
$$
\end{cor}
\begin{proof}
This result is more a corollary of the proof of 
Theorem~\ref{thm:uniform-order-on-grh}
than its statement.  We consider intervals $I_j:=({\rm e}^j,{\rm e}^{j+1}]$ for 
$j\le\log x$, $j$ a non-negative integer.  The sum of reciprocals of
all primes $p\le\exp((\log x)^{1/m})$ is $y/m+O(1)$, so this contribution
to the sum is under control.  We thus may restrict to the consideration
of primes $p\in I_j$ for $j>(\log x)^{1/m}$.  For such an integer $j$,
let $t={\rm e}^{j+1}$.  If $q\mid i_g(p)$ for some prime $q>t^{1/2}\log^2 t$,
then $\ord{g}{p}\le t^{1/2}/\log^2t$, and the number of such primes is
$O(\sum_{k\le t^{1/2}/\log^2t}k\log|ab|)=O(t\log|ab|/\log^4t)$, 
so that the sum of their reciprocals is
$O(\log|ab|/\log^4t)=O((\log x)^{3/m}/j^4)$.  
Summing this for $j>(\log x)^{1/m}$, we get $O(1)$, which is acceptable.

For $J:=(t^{1/2}/\log^2t,t^{1/2}\log^2t]$, with $t={\rm e}^{j+1}$, we
have that the reciprocal sum of the primes $p\in I_j$ with some $q\in J$
dividing $i_g(p)$ (so that $q\mid p-1$)
is $O(\log\log t/\log^2t)=O(\log j/j^2)$.  Summing
this for $j>(\log x)^{1/m}$ is $o(1)$ as $x\to\infty$ and is 
acceptable.

For $q\le t^{1/2}/\log^2t$ we need the GRH.  As in the proof of
Theorem~\ref{thm:uniform-order-on-grh}, the number of primes $p\in I_j$
with $q\mid i_g(p)$ is bounded by a constant times
$$
\frac{t}{\log t}\frac{(q,h)}{q^2}+t^{1/2}\log t.
$$
Thus, the reciprocal sum of these primes $p$ is 
$$
O\left(\frac{(q,h)}{q^2\log t}+\frac{\log t}{t^{1/2}}\right)
=O\left(\frac{(q,h)}{q^2j}+\frac{j}{{\rm e}^{j/2}}\right).
$$
We sum this expression over primes $q$ with $m<q\ll {\rm e}^{j/2}/j^2$ getting
$$
O\left(\frac1{jm\log m}+\frac1j\sum_{q\mid h,~q>m}\frac1q+\frac1{j^2}\right).
$$
Summing on $j\le\log x$ completes the proof.
\end{proof}

\section{Proof of Theorem~\ref{thm:main}}
\label{sec:some-notation}
Let $x$ be large and let $g$ be an integer with $1<|g|\le \log x$.
Define
$$
y = \log\log x, 
\quad m = \lfloor y/\log^3 y \rfloor,
\quad D = m!,
$$ 
and let 
$$
S_k = \{ p \leq x\, :\, (p-1,D) = 2k \}.
$$ 
Then $S_{1}, S_{2}, \ldots, S_{D/2}$ are disjoint sets of
primes whose union equals $\{ 2<p \leq x\}$.  
Let
\begin{equation}
\label{eq:Skalt}
\tilde{S}_k=\left\{p\in S_k\,:\, p\nmid g,~\frac{p-1}{2k}~\Big{|}~\ord{g}{p}\right\}
\end{equation}
be the subset of $S_k$ where $\ord{g}{p}$ is ``large."  Note that if 
$k\le \log y$, $p\in S_k\setminus\tilde{S}_k$, 
and $p\nmid g$, there is some prime
$q>m$ with
$q\mid(p-1)/\ord{g}{p}$, so that $P(i_g(p))>m$.
Indeed, since $k\le\log y$, each prime dividing $D$ also divides
$D/(2k)$, so that $(p-1,D)=2k$ implies that the least prime factor
of $(p-1)/(2k)$ exceeds $m$.

Thus, from Theorem~\ref{thm:uniform-order-on-grh},
$$
|S_k\setminus\tilde{S}_k|\le
|\{ p \leq x : \ord{g}{p} < p/m \}|+\sum_{p\mid g}1 
\ll \frac{\pi(x)}{m}\cdot\frac{h\tau(h)}{\phi(h)}
$$
uniformly for $k\le\log y$.
Using this it is easy to see that $S_{k}$ and $\tilde{S}_k$ are of
similar size when $k$ is small.  However, we shall essentially measure
the ``size" of $S_k$ or $\tilde{S}_k$ by the sum of the 
reciprocals of its members and for this we will use Corollary~\ref{cor:recip}.  
We define
$$
{E}_k := 
\sum_{\substack{p \in {S}_k \\1< p^\alpha \leq x}}
\frac{1}{p^\alpha}
$$
and
$$
\tilde{E}_k := 
\sum_{\substack{p \in \tilde{S}_k \\1< p^\alpha \leq x}}
\frac{1}{p^\alpha}.
$$
By Lemma~1 of \cite{carmichael-lambda}, uniformly for $k \leq \log^2 y$,
\begin{equation}
\label{eq:Ek}
E_{k} = \frac{y}{\log y } \cdot P_{k} \cdot (1+o(1))
\end{equation}
where 
\begin{equation}
\label{eq:Pk}
P_{k} = \frac{{\rm e}^{-\gamma}}{k} \prod_{q>2}
\left(1-\frac{1}{(q-1)^{2}} \right) 
\prod_{q|k,\, q>2} \frac{q-1}{q-2}.
\end{equation}
Note that, with $B$ given by~\eqref{eq:B},
\begin{equation}
\label{eq:Bident}
\sum_{k=1}^\infty\frac{P_k}{2k}=B.
\end{equation}

The following lemma shows that not much is lost when restricting to
primes $p \in \tilde{S}_{k}$.
\begin{lem}
\label{lem:Ek-equals-Ek-tilde}
For $k \leq \log y$, we uniformly have
$$
\tilde{E}_k = E_k \cdot
\left(1 + O \left( \frac{\log^{5} y}{y} \right).
\right)
$$
\end{lem}
\begin{proof}
By \eqref{eq:Ek} and \eqref{eq:Pk}, we have
\begin{equation}
\label{eq:Ek-tilde-lower-bound}
E_k \gg \frac{y}{k \log y}
\geq \frac{y}{ \log^{2} y},
\end{equation}
and it is thus sufficient to show that 
$\sum_{p \in S_k  \setminus   \tilde{S}_k} 1/p \ll \log^3 y$ since
the contribution from prime powers $p^{\alpha}$ for $\alpha \geq
2$ is $O(1)$.   As we have seen, if $k\le\log y$ and 
$p \in S_k  \setminus   \tilde{S}_k$ then
either $p\mid g$ or $P(i_g(p))>m$.
Hence, using Corollary~\ref{cor:recip} and noting that the hypothesis
$|g|\le \log x$ implies that $h\ll y$ and so $h$ has at most one
prime factor $q>m$, we have
$$
\sum_{p \in E_k  \setminus   \tilde{E}_k} \frac1p  
\ll 
\frac{y}{m} = 
\frac{y}{\lfloor y/\log^3 y\rfloor} \ll \log^3 y.
$$
This completes the proof.
\end{proof}

\begin{lem}
\label{lem:sum-Ek-over-k}
We have
$$
\sum_{k\le\log y} \frac{{E}_k}{2k} 
=  \frac{B y}{\log y} (1 + o(1))
$$  
where $B$ is given by \eqref{eq:B}.
\end{lem}
\begin{proof}
This follows immediately from \eqref{eq:Ek}, \eqref{eq:Pk}, and
\eqref{eq:Bident}.
\end{proof}

Given a vector $\bj = (j_{1},j_{2},\ldots,j_{D/2})$
with each $j_{i} \in \Z_{\geq 0}$, let 
$$
\|\bj\| :=
j_{1}+j_{2}+\ldots+j_{D/2}.
$$
Paralleling the notation $\Omega_i(x;\bj)$ from
\cite{carmichael-lambda}, let:
%
%
%
%
\begin{itemize}
\item  $\tilde{\Omega}_{1}(x; {\bf j})$ be the
set of integers that can be formed  by taking products of $v = \|\bj\|$ distinct
primes $p_{1}, p_{2}, \ldots, p_{v}$ in such a way that:
\begin{itemize}
\item for each $i$, $p_{i} < x^{1/y^{3}}$, and
\item the first $j_{1}$ primes are in $\tilde{S}_{1}$, the next $j_{2}$ are in $\tilde{S}_{2}$, etc.;
\end{itemize}

\item  $\tilde{\Omega}_{2}(x; {\bf j})$ be the set of integers
  $u=p_{1}p_{2}\cdots p_{v} \in \tilde{\Omega}_{1}(x; {\bf j})$ such that 
  $(p_{i}-1,p_{j}-1)$ divides $D$ for all $i \neq j$;

\item  $\tilde{\Omega}_{3}(x; {\bf j})$ be the set of integers of the form
  $n=up$ where $u \in  \tilde{\Omega}_{2}(x; {\bf j})$ and $p$
  satisfies $(p-1,D)=2$, $\max(x/2u, x^{1/y}) < p \leq x/u$
  and $\ord{g}{g}{p} > p/y^2$;

\item  $\tilde{\Omega}_{4}(x; {\bf j})$ be the set of integers
$n=(p_{1}p_{2}\cdots p_{v}) p$ in $\tilde{\Omega}_{3}(x; {\bf j})$ with the
additional property that $(p-1,p_{i}-1)=2$ for all $i$.
\end{itemize}

\subsection{Some lemmas}
\label{sec:preliminary-lemmas}

We shall also need the following analogues of Lemmas 2-4 of
\cite{carmichael-lambda}.  
Let
$$
\bJ:=\{\bj:0\leq j_{k}\leq{E}_{k}/k\text{ for $k\le\log y$, 
and $j_{k}=0$ for $k>\log y$}\}.
$$
%
\begin{lem}
\label{lem:two-prime}
If $\bj \in \bJ$, $n \in \tilde{\Omega}_4(x; \bj)$, and $x\ge x_1$, then
$$
\ord{g}{n}
\ge
c_1\frac{x}{y^3 }  
\prod_{k\le\log y} (2k)^{-j_k},
$$
where $x_1, c_1>0$ are absolute constants.
\end{lem}
\begin{proof}
Suppose that  $n = (p_1 p_2 \cdots p_v) p \in \tilde{\Omega}_4(x; \bj)$.
Let $d_i = (p_i-1,D)$, and let $u_i := (p_i-1)/d_i$.  By \eqref{eq:Skalt},
$u_i$ divides $\ord{g}{p_i}$ for all $i$, and by the definition of
$\tilde{\Omega}_3(x;\bj)$ we also have $\ord{g}{p} >p/y^2$.
Since $(p-1)/2$ is coprime to $(p_i-1)/2$ for each $i$ and each
$(p_i-1,p_j-1)\mid D$ for $i\ne j$, we have $u_1,\dots,u_v,p-1$
pairwise coprime.  But
$$
\ord{g}{n} = \lcm[ \ord{g}{p_1}, \ord{g}{p_2}, \ldots, \ord{g}{p_v}, \ord{g}{p}],
$$
so we find that, 
using the minimal order of Euler's function and $\ord{g}{p}> p/y^2$,
\begin{multline*}
\ord{g}{n} \geq u_1 u_2 \cdots u_v \ord{g}{p} \geq
\frac{\phi(n)}
{y^2 \cdot \prod_{i=1}^v d_i } 
\\
\gg \frac{n}{y^2\cdot\log \log n \cdot
 \prod_{k=1}^l (2k)^{j_k}}
\gg \frac{x}{y^3 \cdot  \prod_{k=1}^l (2k)^{j_k}}
\end{multline*}
(recalling that $d_{i} = (p_{i}-1,D) = 2k$ if $p_{i} \in \tilde{S}_k$,
and that $n \in \tilde{\Omega}_4(x; \bj)$ implies that $n>x/2$).
\end{proof}

\begin{lem}
\label{lem:three-prime}
If $\bj \in \bJ$, $u \in \tilde{\Omega}_2(x; \bj)$, and $x\ge x_2$, then
$$
|\{ p: up \in \tilde{\Omega}_4(x; \bj) \}| > c_2x/(uy \log x)
$$
where $x_2, c_2>0$ are absolute constants.
\end{lem}
\begin{proof}
Note that for $\bj\in\bJ$, $\|\bj\|\le
\sum_{k=1}^{l} {E}_k/k \ll y/\log y$ by \eqref{eq:Ek} and \eqref{eq:Pk}.
For such vectors $\bj$,
Lemma 3 of \cite{carmichael-lambda} implies that the number of
primes $p$ with $\max(x/2u,x^{1/y})<p\le x/u$, $(p-1,D)=2$, and $(p-1,p_i-1)=2$
for all $p_i\mid u$ is $\gg x/(uy\log x)$.  Thus it suffices to
show that
$$
|\{p\le x/u:(p-1,D)=2,\,\ord{g}{p}\le p/y^2\}|=o(x/(uy\log x)).
$$
As we have seen, for
$\bj\in\bJ$, $\|\bj\|\ll y/\log y$, so that
$u\in \tilde{\Omega}_2(x;\bj)$ has $u\le x^{1/y^2}$ for all large $x$.
Thus, Theorem~\ref{thm:uniform-order-on-grh} implies that
$$
\sum_{\substack{p\le x/u\\ \ord{g}{p}\le p/y^2}}1
\ll
\frac{\pi(x/u)}{y^2}
\ll
\frac{x}{uy^2 \log x }
=
o\left( \frac{x}{u y \log x}\right).
$$
The result follows.
\end{proof}

\begin{lem}
\label{lem:four-prime}
If $\bj\in\bJ$, then for $x\ge x_3$,  
$$
\sum_{u \in \tilde{\Omega}_2(x; \bj)}
\frac1u
>
\exp \left(
\frac{-c_3 y \log \log y}{\log^2 y}
\right)
\prod_{k\le\log y}
\frac{{E}_k^{j_k}}{j_k!}
$$
where $x_3,c_3>0$ are absolute constants.
\end{lem}
\begin{proof}
The sum in the lemma is equal to
$$
\frac{1}{j_1!j_2!\cdots j_{\lfloor\log y\rfloor}!}
\sum_{\langle p_1,p_2,\dots,p_v\rangle}\frac1{p_1p_2\cdots p_v}
$$
where the sum is over sequences of distinct primes where the first
$j_1$ are in $\tilde{S}_1$, the next $j_2$ are in $\tilde{S}_2$,
and so on, and also each $(p_i-1,p_j-1)\mid D$ for $i\ne j$.  Such
a sum is estimated from below in Lemma~4 of \cite{carmichael-lambda}
but without the extra conditions that differentiate $\tilde{S}_k$ from $S_k$.
The key prime reciprocal sum there is estimated on pages 381--383 to be
$$
E_k\left(1+O\left(\frac{\log\log y}{\log y}\right)\right).
$$
In our case we have the extra conditions that $p\nmid g$ and
$(p-1)/2k\mid \ord{g}{p}$, which alters the sum by a factor
of $1+O(\log^5y/y)$ by Lemma~\ref{lem:Ek-equals-Ek-tilde}.
But the factor $1+O(\log^5y/y)$ is negligible compared with
the factor $1+O(\log\log y/\log y)$, so we have exactly the
same expression in our current case.
The proof is complete.
\end{proof}

\subsection{Conclusion}
\label{sec:conclusion}

For brevity, let $l=\lfloor\log y\rfloor$.
We clearly have
$$
T_g(x)
\geq
\frac{1}{x}
\sum_{\bj \in \bJ}
\sum_{n \in \tilde{\Omega}_4(x; \bj)} 
\ord{g}{n}.
$$
By Lemma~\ref{lem:two-prime}, we thus have
$$
T_g(x)
\gg
\frac{1}{y^3} 
\sum_{\bj \in \bJ}
\prod_{k=1}^{l}
(2k)^{-j_k}
\sum_{n \in \tilde{\Omega}_4(x; \bj)} 
1.
$$
Now,
$$
\sum_{n \in \tilde{\Omega}_4(x; \bj)}  1 
=
\sum_{u \in \tilde{\Omega}_2(x; \bj)}  
\sum_{ up \in \tilde{\Omega}_4(x; \bj)} 1,
$$
and by Lemma~\ref{lem:three-prime}, this is 
$$
\gg
\sum_{u \in \tilde{\Omega}_2(x; \bj)}  
\frac{x}{uy \log x},
$$
which in turn by Lemma~\ref{lem:four-prime} 
is 
$$
\gg
\frac{x}{y \log x}
\exp \left(
\frac{-c_3 y \log \log y}{\log^2 y}
\right)
\prod_{k=1}^l
\frac{{E}_k^{j_k}}{j_k!}.
$$
Hence
$$
T_g(x)
\gg
\frac{x}{y^{4}\log x}
\exp \left(
\frac{-c_3 y \log \log y}{\log^2 y}
\right)
\sum_{\bj \in \bJ}
\prod_{k=1}^{l}
(2k)^{-j_k}
\frac{{E}_k^{j_k}}{j_k!}.
$$
Now,
$$
\sum_{\bj \in \bJ} \prod_{k=1}^{l} (2k)^{-j_k} \frac{{E}_k^{j_k}}{j_k!}
=
\prod_{k=1}^{l}
\left(
\sum_{j_{k}=0}^{[{E}_{k}/k]}
\frac{({E}_k/2k)^{j_k}}{ j_k!}
\right).
$$
Note that $\sum_{j=0}^{2w} w^{j}/j! > {\rm e}^{w}/2$ for $w \geq 1$
and also that ${E}_k/2k \geq 1$ for $x$ sufficiently large, as
${E}_k \gg y/(k \log y)$ by (\ref{eq:Ek-tilde-lower-bound}).  Thus,
$$
\sum_{\bj \in \bJ} \prod_{k=1}^{l} (2k)^{-j_k} \frac{{E}_k^{j_k}}{j_k!}
>
2^{-l} 
\exp \left( 
\sum_{k=1}^l \frac{{E}_k}{2k}
\right).
$$
Hence
$$
T_g(x)
\gg
\frac{x}{y^{4}\log x}
\exp \left(
\frac{-c_3 y \log \log y}{\log^2 y}
\right)
2^{-l} 
\exp \left( 
\sum_{k=1}^l \frac{{E}_k}{2k}
\right).
$$
By Lemma~\ref{lem:sum-Ek-over-k}
we thus have the lower bound in the theorem.
The proof is concluded.

\section{Averaging over prime moduli --- the proofs}

\subsection{Proof of Theorem~\ref{thm:prime-main}}

Let $z=(\log x/\log\log x)^{1/2}$, and abbreviate $\ord{g}{p},i_g(p)$
with $\ell(p),i(p)$, respectively.  We have
$$
\sum_{p\le x}\ell(p)=
\sum_{\substack{p\le x\\ i(p)\le z}}\ell(p)+
\sum_{\substack{p\le x\\ i(p)>z}}\ell(p)
=A+E,
$$
say.
Writing $\ell(p) = (p-1)/i(p)$ and using the identity 
$1/i(p) = \sum_{uv|i(p)} \mu(v)/u$, we find that
%
\begin{align*}
A&=\sum_{\substack{p\le x\\ i(p)\le z}}(p-1)\sum_{uv\mid i(p)}\frac{\mu(v)}{u}\\
&=\sum_{p\le x}(p-1)\sum_{\substack{uv\mid i(p)\\ uv\le z}}\frac{\mu(v)}{u}
-
\sum_{\substack{p\le x\\ i(p)>z}}(p-1)
\sum_{\substack{uv\mid i(p)\\ uv\le z}}\frac{\mu(v)}{u}\\
&= A_1-E_1,
\end{align*}
say.  The main term $A_1$ is
$$
A_1=\sum_{uv\le z}\frac{\mu(v)}{u}
\sum_{\substack{p\le x\\ uv\mid i(p)}}(p-1).
$$
By a simple partial summation using Theorem~\ref{thm:pnt}, the inner sum
here is
$$
\frac12x\frac{\pi(x)}{D_{g}(uv)}
+O\left(\frac{x^2}{\log^2 x}\right),
$$
assuming the GRH.
(By replacing $\frac12x\pi(x)$ with $\pi(x^2)$ or li$(x^2)$, the
error term here can be strengthened to $O(x^{3/2}\log x)$, but
we shall not need this precision.)
Thus,
\begin{align*}
A_1&=\frac12x\pi(x)\left(\sum_{uv\le z}\frac{\mu(v)}{u D_{g}(uv)} \right)
+ O\left(\frac{x^2}{\log^2 x}
\sum_{n\le z}\left|\sum_{uv=n}\frac{\mu(v)}{u}\right|\right).
\end{align*}
The inner sum in the $O$-term is $\phi(n)/n$, so the $O$-term is
$O(x^2z/\log^2x)$.
Recalling that $\rad(n)$ denotes the largest squarefree divisor of
$n$, we note 
that $\sum_{v\mid k}  \mu(v)v = \prod_{p|k} (1-p) =
(-1)^{\omega(k)}\phi(\rad(k))$, and hence 
$$
\sum_{u,v}\frac{\mu(v)}{u D_{g}(uv)}
=\sum_{k\geq 1}\sum_{v\mid k}\frac{\mu(v)v}{D_{g}(k)k}
=\sum_{k\geq 1}\frac{(-1)^{\omega(k)}\phi(\rad(k))}{D_{g}(k)k}
$$
which, on noting that $\phi(\rad(k)) = \phi(k) \rad(k)/k$, equals
$$
\sum_{k\geq 1}\frac{(-1)^{\omega(k)} \rad(k) \phi(k)}{D_{g}(k)k^{2}} = c_{g}
$$
Thus, with $\psi(h):=h\tau(h)/\phi(h)$,
$$
\sum_{uv\le z}\frac{\mu(v)}{uvD_{g}(uv)}=
c_{g}
-\sum_{k > z}\frac{(-1)^{\omega(k)} \rad(k) \phi(k)}{D_{g}(k)k^{2}}
=
c_{g} + O(\psi(h)/z),
$$
by the same argument as in the fourth step of the proof of
Theorem~\ref{thm:uniform-order-on-grh} (in particular, see
(\ref{eq:sum-estimate}).) 
It now follows 
that
$$
A_1= x  \pi(x) \cdot
\left(
\frac{c_{g}}{2}+O(\psi(h) /z)) + O(z/\log x)
\right).
$$

It remains to estimate the two error terms $E,E_1$.
Using Theorem~\ref{thm:uniform-order-on-grh}, we have
$$
E\ll\frac{x}{z}\cdot\frac{\pi(x)}{z}\psi(h)\ll \frac{x\pi(x)\psi(h)}{z^2}.
$$
To estimate $E_1$, we consider separately terms with $z<i(p)\le z^2$
and terms with $i(p)>z^2$, denoting the two sums $E_{1,1},E_{1,2}$,
respectively.  Note that
$$
\left|\sum_{\substack{uv\mid n\\ uv\le z}}\frac{\mu(v)}{u}\right|
\le\sum_{u\mid n}\frac1u\sum_{\substack{v\mid n\\ v\le z}}1\le
\frac{\tau(n)\sigma(n)}{n},
$$
%
where $\sigma(n) = \sum_{d|n} d$.
We use this estimate for $E_{1,1}$, getting
$$
|E_{1,1}|\le \sum_{z<n\le z^2}\frac{\tau(n)\sigma(n)}{n}
\sum_{\substack{p\le x\\ n\mid i(p)}}(p-1)
\ll x\pi(x)\psi(h)\sum_{z<n\le z^2}\frac{\tau(n)\sigma(n)}{n D_{g}(n)},
$$
using Theorem~\ref{thm:uniform-order-on-grh}.  
Since $D_{g}(n)\geq \frac{\phi(n)\cdot n}{2(h,n)}$ by
Proposition~\ref{prop:Wagstaff}, an elementary 
calculation then shows that  
$$
|E_{1,1}|\ll\frac{x\pi(x) \psi(h) \log z }{z}.
$$
For $E_{1,2}$ we use 
$$
\left|\sum_{\substack{uv\mid n\\uv\le z}}\frac{\mu(v)}{u}\right|
\le \sum_{u\le z}\frac1u\sum_{v\le z/u}1\le z\sum_{u\le z}\frac1{u^2}
\ll z.
$$
Thus, using Theorem~\ref{thm:uniform-order-on-grh},
$$
|E_{1,2}|\le xz\sum_{\substack{p\le x\\ i(p)>z^2}}1\ll\frac{x\pi(x)\psi(h)}{z}.
$$

We conclude that
\begin{align*}
\sum_{p\le x}&l(p)=A+E=A_1-E_{1,1}-E_{1,2}+E\\
&=\frac{c_{g}}{2}x\pi(x)+
O\left(\psi(h)\left(\frac{x\pi(x)}{z} + \frac{x \pi(x) z}{\log x} +
\frac{x\pi(x)}{z^2}+\frac{x\pi(x)\log z}{z} \right)\right)\\
&=\frac{c_{g}}{2}x\pi(x)+
O\left(\frac{x^2(\log\log x)^{3/2}\psi(h)}{(\log x)^{3/2}}\right)\\
&=\frac{c_{g}}{2}x\pi(x)+
O\left(\frac{x^2}{(\log x)^{3/2-1/\log\log\log x}}\right),
\end{align*}
using that that $(\log \log x)^{3/2}
\psi(h) \ll (\log x)^{1/\log\log\log x}$ since $h\ll \log x$.  This
completes the proof.

\subsection{Proof of Proposition~\ref{prop:finding-cg}}

\begin{proof}[Proof of Proposition~\ref{prop:finding-cg}]
We begin with the cases $g>0$, or $g<0$ and $e=0$.
Recalling that $D_{g}(k) = \phi(k) k / (\epsilon_g(k)
(k,h))$, we find that
\begin{equation}
\label{eq:finding-c-g}
c_{g}=
\sum_{k\geq 1}\frac{(-1)^{\omega(k)} \rad(k) \phi(k)}{D_{g}(k)k^{2}}
=
\sum_{k\geq 1}\frac{(-1)^{\omega(k)} \rad(k) (k,h) \epsilon_g(k)
  }{k^{3}}.
\end{equation}
Now, since $\epsilon_g(k)$ equals $1$ if $n \nmid k$, and $2$
otherwise,  (\ref{eq:finding-c-g}) equals
\begin{equation}
\label{eq:finding-c-g-two}
\sum_{k\geq 1}\frac{(-1)^{\omega(k)} \rad(k)(h,k)}{k^{3}} +
\sum_{n|k}\frac{(-1)^{\omega(k)} \rad(k) (h,k)}{k^{3}}
=
\sum_{k\geq 1} ( f(k) + f(kn) )
\end{equation}
where the function $f(k) = (-1)^{\omega(k)} \rad(k)(h,k)/k^{3}$ is
multiplicative.

If $p \nmid h$ and $j \geq 1$, we have
$$
f(p^{j}) = -p/p^{3j}.
$$
On the other hand, writing $h = \prod_{p|h} p^{e_{h,p}}$ we have
$$
f(p^{j}) = -p^{1+\min(j,e_{h,p})}/p^{3j}
$$
for $p|h$ and $j \geq 1$.
Since $f$ is multiplicative, 
$$
\sum_{k\geq 1} ( f(k) + f(kn) )
=
\sum_{k~:~\rad(k) | hn} ( f(k) + f(kn) ) 
\cdot \sum_{(k,hn)=1} f(k).
$$
Now, for $p \nmid h$ and $j \geq 1$, we have $f(p^{j}) =
- \rad(p^j)/p^{3j} = -p/p^{3j}$, hence $\sum_{j \geq 0} f(p^{j}) = 1 -
\frac{p}{p^{3}(1-1/p^{3})} = 1-\frac{p}{p^{3}-1}$ and thus
$$
\sum_{(k,hn)=1} f(k) 
=
\prod_{p \nmid hn} F(p)
=
\prod_{p \nmid hn} (1-\frac{p}{p^{3}-1})
=
\frac{c}{\prod_{p | hn} (1-\frac{p}{p^{3}-1})}.
$$
Similarly,
$
\sum_{\rad(k)|hn} f(k) 
=
\prod_{p|hn}
F(p)
$
and 
$$
\sum_{\rad(k)|hn} f(kn) 
=
\prod_{p|hn} \left( \sum_{j \ge e_{n,p}} f(p^{j}) \right)
=
\prod_{p|hn} \left( F(p) - F(p,e_{n,p}) \right).
$$
Hence
\begin{align*}
\sum_{\rad(k)|hn} f(k) &+ \sum_{\rad(k)|hn} f(kn) 
=
\prod_{p|hn}
F(p)
+
\prod_{p|hn} \left( F(p) - F(p,e_{n,p}) \right)\\
&=
\prod_{p|hn}
F(p)\cdot 
\left( 1 + 
  \prod_{p|hn} \left(1 - \frac{F(p,e_{n,p})}{F(p)} \right)
\right).
\end{align*}
Thus
$$
c_{g} = 
\frac{c}{\prod_{p | hn} (1-\frac{p}{p^{3}-1})}
\cdot
\prod_{p|hn}
F(p) \cdot
\left( 1 + 
  \prod_{p|hn} \left(1 - \frac{F(p,e_{n,p})}{F(p)} \right)
\right),
$$
which, by (\ref{eq:Fp-if-no-h}), simplifies to
$$
c_g=
c \cdot
\prod_{p|h}\frac{F(p)}{1-\frac{p}{p^{3}-1}}
\cdot
\left( 1 + 
  \prod_{p|hn} \left(1 - \frac{F(p,e_{n,p})}{F(p)} \right)
\right).
$$

The case $g < 0$ and $e>0$ is similar: using the multiplicativity of
$f$ together with the definition of $\epsilon_{g}(k)$, we find that
\begin{align*}
c_{g}&= 
\sum_{k\geq 1} ( f(k) + f(kn)   )
-\frac{1}{2}
\sum_{j=1}^{e}  \sum_{(k,2)=1} f(2^{j} k)\\
&=
\prod_p F(p) +
\prod_p (F(p) - F(p,e_{n,p})) -
\frac{1}{2} \cdot  (F(2,e+1)-1 ) \cdot \prod_{p>2} F(p) \\
&=
\prod_p F(p) \left(
1+ \prod_{p|n} \left( 1- \frac{ F(p,e_{n,p})}{F(p)} \right)
-
\frac{F(2,e+1)-1}{2F(2)}
\right).
\end{align*}
Again using the fact that 
$$
\prod_p F(p) = \prod_{p \nmid h}
(1-\frac{p}{p^{3}+1}) \prod_{p|h}F(p) = c \cdot \prod_{p|h}
\frac{F(p)}{1-p/(p^3+1)}
$$ 
the proof is concluded.

\end{proof}

\section{Acknowledgments}

Part of this work was done while the authors visited MSRI, as part of
the semester program ``Arithmetic Statistics''.  We thank MSRI for their
support, funded through the NSF.
We are very grateful to Michel Balazard for suggesting Arnold's conjecture
to us.  In addition we thank Pieter Moree for some helpful comments.




\end{document}